\documentclass[12pt,a4paper, twoside]{article}
\usepackage[utf8]{inputenc}
\usepackage{amsmath}
\usepackage{amsfonts}
\usepackage{amssymb}
\usepackage{enumerate}
\usepackage{graphicx}
\usepackage{enumitem}
\usepackage{amsthm}
\usepackage{bbm}
\usepackage{xcolor}
\usepackage{mdframed}
\usepackage{centernot}
\usepackage{doi}
\usepackage{booktabs}
\usepackage[sort&compress, numbers]{natbib}
\usepackage{epigraph}
\usepackage[centering, a4paper]{geometry}
\usepackage{bm}

\newtheorem{lem}{Lemma}[section]
\newtheorem{thm}{Theorem}[section]
\newtheorem*{thm*}{Theorem}
\newtheorem*{result*}{Kubota and Takei \cite{KubotaTakei}, Theorem 3}
\newtheorem{prop}{Proposition}[section]
\newtheorem{cor}{Corollary}[section]

\theoremstyle{definition}

\usepackage{mathtools}  
\usepackage{mathrsfs}   

\definecolor{franceblue}{rgb}{0.19, 0.55, 0.91}
\definecolor{munsell}{rgb}{0.0, 0.5, 0.69}

\numberwithin{equation}{section}

\usepackage{color}  
\usepackage{hyperref}
\hypersetup{
    colorlinks=true, 
    linktoc=all,     
    citecolor=blue, 
    linkcolor=franceblue,  
}

\newcommand{\EE}{\mathbb{E}} 
\newcommand{\var}{\operatorname{Var}} 
\newcommand{\PP}{\mathbb{P}} 

\makeatletter
\newcommand*\dotp{\mathpalette\dotp@{.5}}
\newcommand*\dotp@[2]{\mathbin{\vcenter{\hbox{\scalebox{#2}{$\m@th#1\bullet$}}}}}
\makeatother

\title{Asymptotic Normality of Superdiffusive Step-Reinforced Random Walks}
\author{Marco Bertenghi\footnote{Institute of Mathematics, University of Zurich, Winterthurerstrasse 190, CH-8057 Zürich, Switzerland}}
\date{\today}

\begin{document}
\maketitle
\abstract{In this article we establish for the superdiffusive regime $p \in (1/2,1)$ that the fluctuations of a general \textit{step-reinforced random walk} $\hat{S}$ around $a_n\hat{W}$, where $(a_n)_{n \in \mathbb{N}}$ is a non-negative sequence of order $n^p$ and $\hat{W}$ is a non-degenerate random variable, is Gaussian. This extends a known result by Kubota and Takei in \cite{KubotaTakei} for the \textit{elephant random walk} to the more general setting of step-reinforced random walks. Further, we provide an application of the asymptotic normality of $\hat{S}$ around $a_n \hat{W}$ to reinforced empirical processes as discussed by Bertoin in \cite{BertoinLinear}, which yields a refined Donsker's invariance principle.}
\section{Introduction} \label{section:Introduction}
In short, inspired by an article of Kubota and Takei \cite{KubotaTakei}, the purpose of the present work is to show that the fluctuations of  a (properly rescaled) \textit{random walk with step reinforcement} around its limit is Gaussian. The processes we study belong to particular class of random walks with memory that have been of increasing interest in recent years, see \cite{BertoinLinear}, \cite{BertoinUniversality}. Historically,  the so-called \textit{elephant random
walk} (ERW) has been an  important and fundamental example of a step-reinforced random walk that was originally introduced in the physics literature by Schütz and Trimper \cite{SchuetzTrimper} more than  15 years ago. In order to motivate the present work, we shall first recall the setting of the ERW and also briefly introduce its $d$-dimensional generalisation, the so-called multidimensional ERW (MERW). We then discuss some related results for the ERW before we finally introduce the framework of this article, namely step-reinforced random walks.

The ERW is a one-dimensional discrete-time nearest neighbour random walk with infinite memory, in allusion to the traditional saying that \textit{an elephant never forgets}. It can be depicted as follows: Fix some $q \in (0,1)$, commonly referred to as the \textit{memory parameter}, and suppose that an elephant makes an initial step in $\{-1,1\}$ at time $1$. At each time $n \geq 2$, the elephant recalls uniformly at random a step from its past; with probability $q$, the elephant repeats the remembered step, whereas with complementary probability of $1-q$ it makes a step in the opposite direction to the memory. The elephant is thus more likely to continue walking in the average direction it has already taken when $q>1/2$, whereas for $q<1/2$, the elephant rather tends to walk back. Notably, for $q=1/2$, the elephant has no preference and its path follows that of a simple symmetric random walk on the integer lattice $\mathbb{Z}$.

These heuristics of the ERW naturally translate into higher dimensions: The MERW is a $d$-dimensional nearest-neighbour random walk on the grid $\mathbb{Z}^d$ for some $d \geq 1$. The elephant performs an initial unit-step in $\mathbb{Z}^d$, at any later time, it selects uniformly at random a step from its past and, with probability $q \in (0,1)$, it repeats the remembered step, whereas with probability $(1-q)/(2d-1)$, it moves in one of the remaining $2d-1$ neighbours. For mathematical rigorous definitions of these processes we refer the interested reader to e.g. \cite[Section 3]{BaurBertoin} for the ERW and \cite[Section 2]{Bertenghi} for the MERW.

 The ERW has generated a lot of interest in recent years, a non-exhaustive list of references (with further references therein) is given by \cite{BaurBertoin, Bercu, BercuLaulinCenter, Bertenghi, GavaSchuetzColetti, Coletti2017, Coletti2019, GuevaraERW, KubotaTakei, Kuersten}, see also \cite{BaurClass, GonzalesMult, GuevaraMinimal}, for variations. The asymptotic behaviour after a proper rescaling of the ERW is well-understood. A striking feature that has been observed in \cite{SchuetzTrimper} is that the long-time behaviour of the ERW exhibits a phase transition for $p=1/2$. In light of said remark, it is common in the literature to distinguish between the \textit{diffusive regime} ($p < 1/2$), the \textit{critical regime} ($p=1/2$) and the \textit{superdiffusive regime} ($p>1/2$).

Notably it has been shown in the literature, see \cite[Theorem 3.7, Theorem 3.8]{Bercu}, that in the superdiffusive regime $p \in (1/2,1)$ it holds for $n \to \infty$
\begin{align} \label{firstDisplay}
\frac{S_n^E}{n^p} \to W^E \quad \text{a.s.,}
\end{align}
where $S_n^E$ denotes the position of the ERW at time $n$ and $W^E$ is a non-degenerate non-Gaussian random variable. The a.s. convergence given in (\ref{firstDisplay}) makes it natural to look for a second order weak limit result. In this direction, Kubota and Takei \cite{KubotaTakei} established that the fluctuations of $S_n^E$ around $a_n W^E$, where $a_n$ is a sequence of order $n^p$, are still Gaussian. More precisely, the following result holds:
\begin{result*} \label{thm:ThmOfKubotaTakei} Let $p \in (1/2,1)$. Then there exists a random variable $W^E$ with positive variance such that for the sequence 
\begin{align} \label{seq:CrucialSequence}
a_n =\frac{\Gamma(n+p)}{\Gamma(n)\Gamma(p+1)}, \quad n \geq 1,
\end{align}
it holds that as $n$ tends to infinity
\begin{align*}
\frac{S_n^E- a_n W^E}{\sqrt{n}} \implies \mathcal{N}(0,1/(2p-1)).
\end{align*}
\end{result*}

The treatment of Kubota and Takei is limited to the case of Rademacher distributed steps, that is the framework of the ERW. In this article we will extend the aforementioned result of Kubota and Takei to the more general framework of \textit{step-reinforced random walks} and to arbitrary dimensions.

 A step reinforced random walk is a generalisation of the ERW (respectively MERW for higher dimensions), where the distribution of a typical step of the walk is arbitrary, rather than Rademacher. Vaguely speaking, fix a parameter $p \in (0,1)$, called the \textit{reinforcement parameter}; at each discrete time, with probability $p$ a step reinforced random walk repeats one of its preceding steps chosen uniformly at random, and otherwise, with complementary probability $1-p$, it has an independent increment with a fixed but arbitrary distribution. More precisely, given an underlying probability space $( \Omega, \mathcal{F}, \PP)$ and a sequence $\mathbf{X}_1, \mathbf{X}_2, \dots $ of i.i.d. copies of a $\mathbb{R}^d$-valued random vector $\mathbf{X}$ with law $\mu$ on $\mathbb{R}^d$, all defined on the same given probability space, we define $\mathbf{\hat{X}}_1, \mathbf{\hat{X}}_2, \dots$ recursively as follows: First, let $( \varepsilon_i : i \geq 2)$ be an independent sequence of Bernoulli random variables with parameter $p \in (0,1)$.  We set first $\mathbf{\hat{X}}_1= \mathbf{X}_1$, and next for $i \geq 2$, we let 
\begin{align*}
\mathbf{\hat{X}}_i = \begin{cases} \mathbf{X}_i, & \text{if } \varepsilon_i=0, \\
\mathbf{\hat{X}}_{U[i-1]}, & \text{if } \varepsilon_i=1. \end{cases},
\end{align*}
where $U[i-1]$ denotes an independent uniform random variable on the set $\{1, \dots , i-1\}$. Finally, the sequence of the partial sums 
\begin{align*}
\mathbf{\hat{S}}_n:= \mathbf{\hat{X}}_1 + \dots + \mathbf{\hat{X}}_n, \quad n \in \mathbb{N},
\end{align*}
is referred to as a ($d$-dimensional) \textit{step-reinforced random walk}. In the special case of $d=1$, i.e. when the typical step $X$ is a real-valued random variable, we use the notation 
\begin{align*}
\hat{S}_n = \hat{X}_1 + \dots + \hat{X}_n, \quad n \in \mathbb{N}.
\end{align*} In this setting, when the typical step $X$ follows the Rademacher distribution $\mathcal{R}(1/2)$, Kürsten \cite{Kuersten} pointed out that $\hat{S}$ is a version of the ERW with memory parameter $q=(p+1)/2 \in (1/2,1)$ in the present notation, i.e. we rediscover the ERW in its superdiffusive regime. Similarly, when the typical step $\mathbf{X}$ follows an isotropic law on $\mathbb{R}^d$, then $\mathbf{ \hat{S}}$ is a version of the MERW with parameter $q \in ( (2d)^{-1},1)$.

 Plainly, the position of the step reinforced walker is given by 
\begin{align} \label{rec:defOfRRW}
\mathbf{\hat{S}}_{n+1}=\mathbf{\hat{S}}_n + \mathbf{\hat{X}}_{n+1}.
\end{align}
We henceforth assume that $\mathbf{X}=(X^1, \dots , X^d)^\text{T} \in L^2(\PP)$, meaning that all components $X^1, \dots , X^d$ of $\mathbf{X}$ are square integrable. Further, we denote the covariance matrix of $\mathbf{X}$ by $\bm{\sigma}^2= \EE[ (\mathbf{X} - \EE( \mathbf{X}))( \mathbf{X}- \EE( \mathbf{X}))^\text{T}]$. From the reinforcement algorithm described above, we have for any bounded measurable function $f: \mathbb{R} \to \mathbb{R}_+$, 
\begin{align*}
\EE(f ( \mathbf{ \hat{X}}_{n+1})) = (1-p) \EE(f( \mathbf{X}_{n+1})) + \frac{p}{n} \sum_{j=1}^n \EE( f( \mathbf{ \hat{X}}_j))
\end{align*}
which yields by induction that each reinforced step $\mathbf{\hat{X}}_n$ has law $\mu$.  In particular the map $n \mapsto \EE( \mathbf{ \hat{X}}_n)= \EE( \mathbf{X})$ is constant. Notice that if $( \mathbf{ \hat{S}}_n)_{n \in \mathbb{N}}$ is not centred, it is often fruitful to reduce our analysis to the centred case by considering $( \mathbf{ \hat{S}}_n- \EE(\mathbf{ \hat{S}}_n))_{n \in \mathbb{N}} = ( \mathbf{ \hat{S}}_n - n \EE( \mathbf{X}))_{n \in \mathbb{N}}$, which is a centred step-reinforced random walk with typical step distributed as $\mathbf{X}- \EE( \mathbf{X})$. We shall adapt the terminology used for the ERW, namely we refer to $p \in (1/2,1)$ as the superdiffusive regime.

 It is our main goal to establish the following fluctuation result:
\begin{thm}\label{thm:MyMainResult} Suppose that $\mathbf{X} \in L^2(\PP)$ and that $p \in (1/2,1)$. Let $( \mathbf{ \hat{S}}_n)_{n \geq 1}$ be a ($d$-dimensional) step-reinforced random walk. Then there exists a non-degenerate random vector $ \mathbf{\hat{W}} \in L^2( \PP)$ in $\mathbb{R}^d$ such that we have the almost sure convergence
\begin{align*}
\frac{\mathbf{ \hat{S}}_n - n \EE( \mathbf{X})}{n^p} \to \mathbf{ \hat{W}}, \quad \text{as } n \to \infty.
\end{align*}
Furthermore, let $(a_n)_{n \in \mathbb{N}}$ be as in (\ref{seq:CrucialSequence}), then there is the convergence in distribution as $n$ tends to infinity,
\begin{align*}
\frac{ \mathbf{\hat{S}}_n- n \EE( \mathbf{X}) - a_n\mathbf{\hat{W}}}{\sqrt{n}} \implies \mathcal{N}_d(\bm{0}, \bm{\sigma}^2/(2p-1)).
\end{align*}
\end{thm}
To summarise: In comparison to the work of Kubota and Takei, we extend their result to the more general framework of step-reinforced random walks. Most notably, even in the one-dimensional case we don't assume that the steps follow a simple distribution such as Rademacher, a fact that has been crucially exploited by Kubota and Takei in their article. Our approach here relies, just as in \cite{KubotaTakei}, on a martingale closely related to $\hat{S}$ and a martingale central limit theorem due to Heyde. Since in our framework the steps of the reinforced random walk can be unbounded, in order to establish Theorem \ref{thm:MyMainResult}, we first need to work with bounded steps and then later on remove this assumption by a truncation argument reminiscent of \cite[Section 4.3]{BertoinUniversality}.

The rest of the paper is organized as follows: Section \ref{sec:Section2} is devoted to the proof of Theorem \ref{thm:MyMainResult}, which is carried out through four subsections. In Section \ref{subsection:AuxiliaryMartingale} we introduce and investigate a crucial martingale. During Sections \ref{subsection:BoundedCase} and \ref{subsection:Truncation} we carry out the main work of the proof. In Section \ref{subsection:ReductionTod1} we finalise the proof of Theorem \ref{thm:MyMainResult} by reducing the statement to the one-dimensional case. In Section \ref{section:EmpiricalProcesses} we present an application of Theorem \ref{thm:MyMainResult}, to empirical processes associated to the aforementioned reinforcement algorithm. Bertoin \cite{BertoinLinear} discusses the question of how empirical processes are affected by  reinforcement with respect to Donsker's theorem and provides invariance principles for all regimes $p \in (0,1)$. We present a refined invariance principle for the superdiffusive regime $p \in (1/2,1)$, see Corollary \ref{thm:MySecondMainResult}. 
\section{Proof of Theorem \ref{thm:MyMainResult}} \label{sec:Section2}
The purpose of this section is to establish Theorem \ref{thm:MyMainResult}. Without loss of generality we shall henceforth assume that $\mathbf{X} \in L^2(\PP)$ is centered, i.e. $\EE( \mathbf{X})= \mathbf{0}$ and $\bm{\sigma}^2= \EE( \mathbf{X} \mathbf{X}^\text{T})$ is the covariance matrix of $\mathbf{X}$. Further, for the sake of presentation, we shall always stick to the boldface notation when we discuss the situation in higher dimensions (i.e. $d \geq 2$), in this direction $\mathbf{X}$ denotes a typical step of the ($d$-dimensional) step-reinforced random walk $\mathbf{ \hat{S}}=( \mathbf{ \hat{S}}_n)_{n \geq 1}$ whereas $X$ is the typical step of the (one-dimensional) step-reinforced random walk $\hat{S}=( \hat{S}_n)_{n \geq 1}$. Similarly, $\mathbf{ \hat{W}}$ denotes a (non-degenerate) random vector in $\mathbb{R}^d$, whereas $\hat{W}$ denotes a (non-degenerate) random variable. 

In order to establish Theorem \ref{thm:MyMainResult} we will work in multiple stages, in fact, most of the work is done in dimension one and then generalised to higher dimensions at the end. More precisely: First, in Section \ref{subsection:AuxiliaryMartingale}, we introduce a crucial auxiliary martingale and investigate its properties. This martingale is essential for our cause in order to prove Theorem \ref{thm:MyMainResult}. In Section \ref{subsection:BoundedCase} we present a proof of Theorem \ref{thm:MyMainResult} in dimension one, first under the additional assumption that the typical step $X$ is bounded, later on, in Section \ref{subsection:Truncation}, this assumption is relaxed to the general case by a truncation argument. Finally, in Section \ref{subsection:ReductionTod1}, we show how Theorem \ref{thm:MyMainResult} can be reduced to the one-dimensional case with an appeal to the Cramér-Wold theorem. 

For convenience and as it is central for our purpose, we recall a Theorem by Heyde \cite{Heyde} which we frequently refer to during the proof of Theorem \ref{thm:MyMainResult}.
\begin{thm}[Heyde \cite{Heyde}, Theorem 1 (b)] \label{thm:Heyde} Suppose that $(M_n)_{n \geq 1}$ is a square-integrable martingale with mean zero. Let $d_k = M_k-M_{k-1}$ for $k=1,2, \dots ,$ where $M_0=0$ almost surely. If 
\begin{align*}
\sum_{k=1}^\infty \EE[(d_k)^2] < + \infty
\end{align*}
holds in addition, then we have the following: Let 
\begin{align*}
s_n^2:= \sum_{k=n}^\infty \EE[ (d_k)^2].
\end{align*}
\begin{enumerate}[label=\roman*)]
\item The limit $M_\infty := \sum_{k=1}^\infty d_k$ exists almost surely, and $M_n \to M_\infty$ in $L^2$. 
\item Assume that 
\begin{enumerate}
\item $\displaystyle \frac{1}{s_n^2}\sum_{k=n}^\infty (d_k)^2 \to 1$ as $n$ tends to infinity in probability and,
\item $\displaystyle \lim_{n \to \infty} \frac{1}{s_n^2}\EE \left( \sup_{k \geq n} (d_k)^2 \right) =0.$
\end{enumerate}
Then we have 
\begin{align*}
\frac{M_\infty-M_n}{s_{n+1}}= \frac{\sum_{k=n+1}^\infty d_k}{s_{n+1}} \implies \mathcal{N}(0,1), \quad \text{as } n \to \infty.
\end{align*}
\end{enumerate}
\end{thm}
\subsection{An auxiliary martingale} \label{subsection:AuxiliaryMartingale}
We start by proving an auxiliary result which will be essential in order to prove Theorem \ref{thm:MyMainResult}. 
\begin{prop}\label{prop:MultMartStart} Let $p \in (1/2,1)$. The process $M_n = \hat{S}_n/a_n$ for $n \geq 1$, where $(a_n)_{n \geq 1}$ is given by (\ref{seq:CrucialSequence}), is a
 martingale with respect to the filtration $\mathcal{F}_n = \sigma (\hat{X}_1, \dots , \hat{X}_n)$ with mean zero. Further, there exists a non-degenerate random variable $\hat{W}$ such that for we have $M_n \to \hat{W}$ a.s. and in $L^2(\PP)$ as $n \to \infty$.
\end{prop}
\begin{proof}
We first point out that for any $n \in \mathbb{N}$, we have 
\begin{align*}
\EE \left( \hat{X}_{n+1} \mid \mathcal{F}_n \right) = (1-p) \EE (X) + p \frac{\hat{X}_1 + \dots + \hat{X}_n}{n} = p \frac{\hat{S}_n}{n}.
\end{align*}
Hence, by (\ref{rec:defOfRRW}) and if we set $\gamma_n = n/(n+p)$, then
\begin{align}\label{proofstep:NeededForMartingale}
\EE( \hat{S}_{n+1} \mid \mathcal{F}_n) =  \frac{\hat{S}_n}{\gamma_n}.
\end{align} 
Moreover, 
\begin{align*}
\prod_{k=1}^{n-1} \gamma_k = \frac{\Gamma(n) \Gamma(p+1)}{ \Gamma(n+p)} = \frac{1}{a_n}
\end{align*}
where $\Gamma$ stands for the Euler gamma function. Therefore, let $(M_n)_{n \geq 1}$ be the sequence of random variables defined, for all $n \geq 1$, by $M_n= \hat{S}_n/a_n$. With an appeal to (\ref{proofstep:NeededForMartingale}), we obtain
\begin{align*}
\EE(M_{n+1} \mid \mathcal{F}_n)= \frac{1}{a_{n+1}} \EE(  \hat{S}_{n+1} \mid \mathcal{F}_n) = \frac{\gamma_n}{a_n}\EE(\hat{S}_{n+1} \mid \mathcal{F}_n) =  \frac{\hat{S}_n}{a_n}=M_n.
\end{align*}
Hence $(M_n)_{n \geq 1}$ is a martingale. Further, as $X$ is centered, we have 
\begin{align*}
\EE( M_n) = \EE( \hat{X}_1)= \EE(X_1)=\EE(X)=0.
\end{align*}
Next, we observe that 
\begin{align}
\EE(M_{n+1}^2-M_n^2 \mid \mathcal{F}_n) &= \EE( (M_{n+1}-M_n)^2 \mid \mathcal{F}_n) \notag \\
&= \frac{\EE ( (\hat{X}_{n+1}- \EE(\hat{X}_{n+1} \mid \mathcal{F}_n))^2 \mid \mathcal{F}_n )}{a_{n+1}^2} \notag \\
&= \frac{\EE(\hat{X}_{n+1}^2 \mid \mathcal{F}_n)- ( \EE(\hat{X}_{n+1} \mid \mathcal{F}_n))^2}{a_{n+1}^2} \notag \\
&= \frac{\EE(\hat{X}_{n+1}^2 \mid \mathcal{F}_n) - \frac{p^2}{n^2}\hat{S}_n^2}{a_{n+1}^2} \notag \\
&= \frac{\EE(\hat{X}_{n+1}^2 \mid \mathcal{F}_n)}{a_{n+1}^2}-\frac{p^2}{n^2}M_n^2. \label{eq:RequiredForProof}
\end{align}
In particular we obtain 
\begin{align}
\EE( M_{n+1}^2- M_n^2) &= \frac{\EE(\hat{X}_{n+1}^2)}{a_{n+1}^2}- \frac{p^2}{n^2}\EE(M_n^2) \notag \\
&= \frac{\EE(X^2)}{a_{n+1}^2}- \frac{p^2}{n^2}\EE(M_n^2) \notag \\
&= \frac{\sigma^2}{a_{n+1}^2}- \frac{p^2}{n^2}\EE(M_n^2)\label{ineq:summableneed} \\
& \leq \frac{\sigma^2}{a_{n+1}^2}. \notag
\end{align}
As a consequence, by summarizing the most left side of the inequality above over all $k=0, \dots , n$, we obtain that the martingale $(M_n)_{ n \in \mathbb{N}}$ is square integrable. 

Next, let us denote by $d_k:=M_k-M_{k-1}$, which is a martingale difference for $k\in \mathbb{N}$.
Thanks to (\ref{ineq:summableneed}) we have for all $k=1,2, \dots$
\begin{align*}
\EE((d_k)^2 )\leq \frac{\sigma^2}{a_{k}^2}.
\end{align*}
Further, by Stirling's formula for Gamma functions, we have 
\begin{align*}
a_n \sim \frac{n^p}{\Gamma(p+1)}, \quad \text{as } n \to \infty.
\end{align*}
Hence, 
\begin{align} \label{Asymptotic:UsefulForSecondMoment}
\EE(|d_n|^2) \leq \frac{ \sigma^2}{a_n^2}  \sim \sigma^2 \frac{\Gamma(p+1)^2}{n^{2p}}, \quad \text{as } n \to \infty,
\end{align}
which is summable as $p>1/2$. Thus Theorem \ref{thm:Heyde} i) of Heyde implies that \begin{align*}
\hat{W}:= \sum_{k=1}^\infty d_k = \lim_{n \to \infty} M_n = \lim_{n \to \infty} \frac{\hat{S}_n}{a_n}
\end{align*}
exists almost surely, and $M_n \to \hat{W}$ almost surely and in $L^2$ as $n \to \infty$. Plainly, we have $\EE(\hat{W})=0$ and 
\begin{align*}
\EE(\hat{W}^2)= \sum_{k=1}^\infty \EE((d_k)^2)>0.
\end{align*}
Thus $\hat{W}$ is of positive variance and therefore non-degenerate. Further it follows from (\ref{Asymptotic:UsefulForSecondMoment}) that $\hat{W} \in L^2( \PP)$. This concludes the proof.
\end{proof}
As a consequence of Proposition \ref{prop:MultMartStart} we recover and improve Theorem 3.2. in \cite{BertoinUniversality}. In particular we establish a stronger convergence as described in Theorem \ref{thm:MyMainResult} for dimension one. 
\begin{cor} \label{cor:ImprovementOfBertoinRes} Let $p \in (1/2,1)$, and suppose $X \in L^2(\PP)$, then we have
\begin{align*}
\lim_{n \to \infty} \frac{\hat{S}_n}{n^p} = L,
\end{align*}
a.s. and in $L^2(\PP)$ where $L \in L^2(\PP)$ is some non-degenerate random variable. In particular 
\begin{align*}
\lim_{n \to \infty} \frac{\hat{S}_n}{n}= 0 \quad \text{a.s.}
\end{align*}
\end{cor}
\begin{proof}
Thanks to Proposition \ref{prop:MultMartStart} we obtain that 
\begin{align*}
\frac{\hat{S}_n}{a_n} \xrightarrow{n \to \infty} \hat{W}
\end{align*}
a.s. and in $L^2(\PP)$. Since $a_n \sim n^p/\Gamma(p+1)$ as $n \to \infty$, the claim follows by letting $L:=\hat{W}/\Gamma(p+1)$. Since $\hat{W} \in L^2(\PP)$, it follows immediately that $L \in L^2(\PP)$. The second statement then follows plainly as $\hat{L} \in L^2(\PP)$ and
\begin{align*}
\frac{\hat{S}_n}{n}= \frac{\hat{S}_n}{n^p} \frac{1}{n^{1-p}} \to 0 \quad \text{a.s.}
\end{align*}
\end{proof}
\subsection{The case when $X$ is bounded} \label{subsection:BoundedCase}
In this section we shall present a proof of Theorem \ref{thm:MyMainResult} in dimension $d=1$. To make the proof more tractable, we shall first do this under the additional assumption that the typical step $X$ is bounded. In this direction, we first recall a useful lemma from calculus.
\begin{lem}[Heyde \cite{Heyde}, Lemma 1 (ii)] \label{lem:KroneckersLem} Consider a positive real sequence $(a_n)_{n \geq 1}$ which monotonically diverges to $+ \infty$, and let $(b_n)_{n \geq 1}$ be another real-valued sequence. If $\sum_{k=1}^\infty a_k b_k < \infty,$
then 
\begin{align*}
\lim_{n \to \infty} a_n \sum_{k=n}^\infty b_k =0. 
\end{align*}

\end{lem}
\begin{proof}[Proof of Theorem \ref{thm:MyMainResult}]
Since $M_n \to \hat{W}$ in $L^2$ and thanks to (\ref{ineq:summableneed}) we have
\begin{align*}
\EE[(d_n)^2] = \frac{\sigma^2}{a_{n}^2}- \frac{p^2}{(n-1)^2}\EE(M_{n-1}^2) \sim \frac{\sigma^2 \Gamma(p+1)^2}{n^{2p}}, \quad \text{as } n \to \infty.
\end{align*}
Hence
\begin{align} 
s_n^2 &:= \sum_{k=n}^\infty \EE[(d_k)^2] \notag \\ &\sim \sigma^2 \Gamma(p+1)^2 \sum_{k=n}^\infty \frac{1}{k^{2p}} \notag \\
& \sim \frac{\sigma^2 \Gamma(p+1)^2}{2p-1} \frac{1}{n^{2p-1}} \notag \\
& \sim \frac{\sigma^2}{2p-1}n \frac{1}{a_n^2}, \quad \text{as } n \to \infty.\label{eq:bound}
\end{align}
Let $\hat{V}_n := \hat{X}_1^2 + \dots + \hat{X}_n^2$, then $(\hat{V}_n)_{n \geq 1}$ is another step-reinforced random walk with typical step distributed as $X^2$ and it holds that
\begin{align*}
\EE( \hat{X}_{k+1}^2 \mid \mathcal{F}_k) = p \frac{\hat{V}_k}{k}+(1-p) \sigma^2.
\end{align*}
Observe that $X^2$ is no longer centered, however, we introduce 
\begin{align*}
 \hat{W}_n &= \left( \hat{X}_1^2 - \EE(X^2) \right) + \dots + \left( \hat{X}^2_n- \EE(X^2) \right) \\
     &= \hat{Y_1} + \dots + \hat{Y_n}
\end{align*}
with an obvious choice of notation. In turn, this is a step-reinforced version of the random walk with typical step distributed as $X^2- \EE(X^2)$ and hence centered. Since 
\begin{align*}
\hat{V}_n = \hat{W}_n + n \EE(X^2) = \hat{W}_n + n \sigma^2
\end{align*} and because $p \in (1/2,1)$, we then obtain by Corollary \ref{cor:ImprovementOfBertoinRes} that 
\begin{align*}
\lim_{n \to \infty} \frac{\hat{V}_n}{n}= \sigma^2 \qquad \text{a.s.}
\end{align*}
As the martingale $(M_n)_{n \geq 1}$ converges in $L^2 ( \PP)$ to $W$, we conclude with an appeal to (\ref{eq:RequiredForProof}) that as $n$ tends to infinity 
\begin{align*}
\EE[( d_{n+1})^2 \mid \mathcal{F}_n) \sim \frac{\sigma^2\Gamma(p+1)^2}{n^{2p}}, \quad \text{a.s.}
\end{align*}
A computation analogous to (\ref{eq:bound}) yields as $n \to \infty$
\begin{align*} 
A_n^2:= \sum_{k=n}^\infty \EE[ ( d_k)^2 \mid \mathcal{F}_{k-1}] \sim \frac{\sigma^2}{2p-1}n \frac{1}{a_n^2}, \quad \text{a.s.},
\end{align*}
in particular this implies that 
\begin{align} \label{eq:importantlimit}
\lim_{n \to \infty} \frac{A_n^2}{s_n^2}=1 \quad  \text{in probability}.
\end{align}
Further, as $n \to \infty$, 
\begin{align*}
\sum_{k=1}^n \EE\left [ (d_k)^2 \mid \mathcal{F}_{k-1} \right] \sim \sigma^2 \Gamma(p+1)^2 \sum_{k=1}^n \frac{1}{k^{2p}}
\end{align*}
and as $p \in (1/2,1)$ the above is finite for $n \to \infty$. Hence we have 
\begin{align*}
\sum_{k=1}^\infty \EE \left[ (d_k)^2 \mid \mathcal{F}_{k-1} \right] < \infty \text{ a.s.}
\end{align*}
As $(M_n)_{n \geq 1}$ is a martingale which converges a.s. we conclude that
\begin{align*}
\sum_{k=1}^\infty \frac{1}{s_k^2} \left[ (d_k)^2- \EE\left[ (d_k)^2 \mid \mathcal{F}_{k-1} \right] \right] < + \infty \text{ a.s.}
\end{align*}
Since plainly $(1/ s_n^2)_{n \geq 1}$ monotonically diverges to $+\infty$, the above yields together with Lemma \ref{lem:KroneckersLem} that
\begin{align*}
\lim_{n \to \infty} \frac{1}{s_n^2}\sum_{k=n}^\infty \left[ (d_k)^2- \EE \left[ (d_k)^2 \mid \mathcal{F}_{k-1} \right] \right] =0 \text{ a.s.}
\end{align*}
and together with (\ref{eq:importantlimit}) we conclude that 
\begin{align} \label{ConditionAHeyde}
\lim_{n \to \infty} \frac{1}{s_n^2}\sum_{k=n}^\infty (d_k)^2 =1 \quad \text{in probability}.
\end{align}
In turn, (\ref{ConditionAHeyde}) shows that condition a) of Theorem \ref{thm:Heyde} of Heyde holds.

Next, we have for $k=1,2, \dots $
\begin{align}
d_k
&= \frac{\hat{S}_k - \EE( \hat{S}_k \mid \mathcal{F}_{k-1})}{a_k} \notag \\
&= \frac{\hat{X}_k- \EE( \hat{X}_k \mid \mathcal{F}_{k-1})}{a_k} \label{eq:MakesItEasier}\\
&= \frac{\hat{X}_k}{a_k}- \frac{p}{a_k} \frac{\hat{S}_{k-1}}{k-1} \notag .
\end{align}
Under our standing assumption that the typical step $X$ of $( \hat{S}_n)_{n \geq 1}$ is bounded, we observe from (\ref{eq:MakesItEasier}) that 
\begin{align*}
|d_k| \leq  \frac{2 \|X \|_\infty}{a_k} \leq 2 \| X\|_\infty.
\end{align*}
 We have 
\begin{align} \label{bound:SimpleWhenBounded}
\sup_{k \geq n } d_k^2  \leq \frac{4 \|X\|_\infty^2}{a_n^2} \sim \frac{4 \|X\|_\infty^2 \Gamma(p+1)^2}{n^{2p}}, \quad \text{as } n \to \infty. 
\end{align}
Using (\ref{bound:SimpleWhenBounded}) and (\ref{eq:bound}) yields that 
\begin{align} \label{ConditionBHeyde}
\lim_{n \to \infty} \frac{1}{s_n^2} \EE \Big( \sup_{k \geq n } d_k^2 \Big) \leq \lim_{n \to \infty} \frac{1}{s_n^2} \frac{4 \|X\|_\infty^2}{a_n^2} =  \frac{4(2p-1) \|X\|_\infty^2}{\sigma^2} \lim_{n \to \infty} \frac{1}{n} =0.
\end{align}
With (\ref{ConditionBHeyde}) we conclude that condition b) of Theorem \ref{thm:Heyde} is satisfied.

Since
\begin{align*}
\frac{\hat{W}-M_n}{s_{n+1}} = \frac{a_n\hat{W}- \hat{S}_n }{s_{n+1}a_n}
\end{align*}
and because
\begin{align*}
s_{n+1}  \sim \sqrt{ \frac{\sigma^2}{2p-1}n} \frac{1}{a_n}, \quad \text{as }n \to \infty
\end{align*}
the desired conclusion follows from Theorem \ref{thm:Heyde} as
\begin{align*}
\frac{a_n \hat{W}-\hat{S}_n}{ \sqrt{\frac{\sigma^2}{2p-1} n}} \sim \frac{\hat{W}-M_n}{s_{n+1}}   \implies \mathcal{N}(0,1), \quad \text{as } n \to \infty.
\end{align*}
Hence the proof of Theorem \ref{thm:MyMainResult} is complete when $\|X \|_\infty < \infty$. 
\end{proof}
So far we have established in dimension $d=1$ and under the additional assumption that the typical step $X \in L^2(\PP)$ is bounded the distributional convergence as $n \to \infty$
\begin{align} \label{question:Replacement}
\frac{\hat{S}_n-a_n \hat{W}}{\sqrt{n}} \implies \mathcal{N}(0, \sigma^2/(2p-1)).
\end{align}
One can wonder whether it is possible to replace $a_n$ with $n^p$ in (\ref{question:Replacement}), as the two expressions are asymptotically equivalent. However, this is not possible because
\begin{align*}
 \frac{n^p-a_n}{\sqrt{n}} = n^{p- 1/2} \left( 1- \frac{a_n}{n^p} \right) \sim n^{p-1/2} \left( 1- \Gamma(p+1) \right), \quad \text{as } n \to \infty
\end{align*}
and since $p \in (1/2,1)$ the term $n^{p-1/2}$ explodes to infinity as $n \to \infty$. 
\subsection{A truncation argument} \label{subsection:Truncation}
In this section we only assume that $X \in L^2(\PP)$ is centered. We shall now complete the proof of Theorem \ref{thm:MyMainResult} for dimension $d=1$ by a truncation argument, this idea is adapted from Section 4.3 of \cite{BertoinUniversality}. 

We require the following technical lemma concering convergence on metric space: 
 \begin{lem} \label{lemma:theReductionLemma} Let $(E,d)$ be a metric space and consider  $(a_n^{(m)} \, : \, m,n \in \mathbb{N})$ a family of sequences, with  $a_n^{(m)} \in E$ for all $n, m \in \mathbb{N}$. Suppose further that the following conditions are satisfied:
\begin{enumerate}
    \item[(i)] For each fixed $m \in \mathbb{N}$, $a_n^{(m)} {\rightarrow} a_\infty^{(m)}$ as $n \to \infty$ for some element $a_\infty^{(m)} \in E$.
    \item[(ii)] $a_\infty^{(m)} {\rightarrow} a_{\infty}^{(\infty)}$ as  $m \to \infty$, for some $a_\infty^{(\infty)} \in E$.
\end{enumerate}
Then, there exists a non-decreasing  subsequence $(b_n)_{n \in \mathbb{N}}$ with $b_n \rightarrow  \infty$ as $n \to \infty$,  for which the following convergence holds:
\begin{equation*}
    a_n^{(b_n)} {\rightarrow} a_\infty^{(\infty)} \quad \text{ as } n \to \infty.
\end{equation*}
\end{lem}
\begin{proof}
Since the sequence $(a_\infty^{(m)})_m$ converges, we can find an increasing subsequence $m_1 \leq m_2 \leq \dots $ satisfying 
\begin{equation*}
    d(a_\infty^{(m_k)} , a_\infty^{(m_{k+1})}  ) \leq 2^{-k} \quad \quad \text{ for each } k \in \mathbb{N}.
\end{equation*}
Moreover, since for each fixed $m_k$ the corresponding sequence $(a_n^{(m_k)})_n$ converges, there exists a strictly increasing sequence $(n_k)_k$ satisfying that, for each $k \in \mathbb{N}$, 
\begin{align*}
     d(a_i^{(m_k)} , a_\infty^{(m_k)} ) \leq 2^{-k} \quad \quad \text{ for all } i \geq n_k.
\end{align*}
Now, we set
\begin{align*}
b_n:=\begin{cases} m_1, & \text{if } n < n_1, \\ m_k, & \text{if } n_k \leq n < n_{k+1} \end{cases}.
\end{align*}
Observe that for each $n \in \mathbb{N}$ there always exists an appropriate $k^* \in \mathbb{N}$ such that $n_{k^*} \leq n < n_{k^*+1}$. 
Plainly, $(b_n)_{n \in \mathbb{N}}$ is non-decreasing and satisfies $b_n \to \infty$ as $n \to \infty$. Finally, we claim $(a_n^{(b_n)})_n$ converges to $a_\infty^{(\infty)}$ as $n \to \infty$. Indeed, it suffices to observe that for $n_k \leq n < n_{k+1}$, by the triangle-inequality
\begin{align*}
    d(a_n^{(b_n)} , a_\infty^{(\infty)})      
    \leq d(a_n^{(m_k)} , a_\infty^{(m_k)}) + d(a_\infty^{(m_k)}, a_\infty^{(\infty)}) \leq  2^{-k} + \sum_{i=k}^\infty 2^{-i}.
\end{align*}
We conclude the proof by letting $n$ tend to infinity. 
\end{proof}
We are now ready to demonstrate the reduction argument. In this direction we introduce for any $b>0$
\begin{align*}
X^{(b)}:= X \mathbf{1}_{\{|X|\leq b\}}- \EE(X \mathbf{1}_{\{ |X| \leq b \}}),
\end{align*}
hence $X^{(b)}$ is a bounded and centered random variable, we shall denote by $\sigma^{(b)}= \sqrt{ \var (X^{(b)})}$ its standard deviation. Similarly, we set
\begin{align*}
\hat{X}_n^{(b)}:= \hat{X}_n \mathbf{1}_{\{ | \hat{X}_n| \leq b\}}- \EE( X \mathbf{1}_{\{ |X| \leq b \}}),
\end{align*}
and
\begin{align*}
\hat{S}_n^{(b)} = \hat{X}_1^{(b)} + \dots + \hat{X}_n^{(b)}.
\end{align*}
Clearly, $\hat{S}^{(b)}$ is a version of the step-reinforced random walk with typical step distributed as $X^{(b)}$. As the latter is bounded, an application of Theorem \ref{thm:MyMainResult} as proven in the last section shows that we have for $\hat{W}^{(b)}= \lim_{n \to \infty} \hat{S}_n^{(b)}/a_n$
\begin{align*}
\frac{\hat{S}_n^{(b)}- a_n\hat{W}^{(b)}}{\sqrt{n}} \implies \frac{\sigma^{(b)}}{\sqrt{2p-1}} \mathcal{N}(0,1), \quad \text{as } n \to \infty. 
\end{align*}
Since plainly $\sigma^{(b)}$ converges to $\sigma$ as $b \to \infty$ and since the convergence in distribution is metrisable, it follows readily by Lemma \ref{lemma:theReductionLemma} that there exists an increasing sequence $(b_n)_{n \geq 1}$ tending towards infinity such that
\begin{align*}
\frac{\hat{S}_n^{(b_n)}-a_n \hat{W}^{(b_n)}}{\sqrt{n}} \implies \frac{\sigma}{\sqrt{2p-1}}\mathcal{N}(0,1), \quad \text{as } n \to \infty,
\end{align*}
 Observe that $\hat{W}^{(1)}:= \hat{W}^{(b_n)}= \lim_{n \to \infty} \hat{S}_n^{(b_n)}/a_n$ is independent of $n$.  

Next, we consider 
\begin{align*}
\check{S}_n^{(b_n)}:= \hat{S}_n- \hat{S}_n^{(b_n)}
\end{align*}
and observe that 
\begin{align*}
\check{S}_n^{(b_n)} = \check{X}_1^{(b_n)} + \dots + \check{X}_n^{(b_n)}
\end{align*}
where
\begin{align*}
\check{X}_n^{(b_n)} = \hat{X}_n \mathbf{1}_{\{ | \hat{X}_n| > b_n\}}- \EE( \hat{X} \mathbf{1}_{\{ |\hat{X}| > b_n\}}).
\end{align*}
In turn, $\check{S}^{(b_n)}$ is also a step-reinforced random walk, this time with typical step distributed as $X-X^{(b_n)}$. Since $X$ is centered in $L^2(\PP)$, so is $X-X^{(b_n)}$, and if we write $\zeta^{(b_n)}= \sqrt{ \var (X-X^{(b_n)})}$ for its standard deviation, then clearly $\lim_{n \to \infty} \zeta^{(b_n)} =0$ since $b_n \to \infty$ as $n \to \infty$. Furthermore, if we set $\check{W}^{(b)}=\lim_{n \to \infty} \check{S}_n^{(b)}/a_n$ and $\check{W}^{(2)}= \check{W}^{(b_n)}$ then we have for $\hat{W}:= \hat{W}^{(1)}+\check{W}^{(2)}$, $Y^n:= n^{-1/2}( \hat{S}_n^{(b_n)}- \hat{W}^{(b_n)}a_n)$ and $Z^n:= n^{-1/2} ( \check{S}_n^{(b_n)}- \check{W}^{(b_n)}a_n)$ the decomposition 
\begin{align}
n^{-1/2} (\hat{S}_n-a_n \hat{W}) = Y^n + Z^n. \label{eq:decomposition}
\end{align}
For $k \geq n$ we consider the martingale $M'_k:= \check{S}_k^{(b_k)}/a_k$. An application of Proposition \ref{prop:MultMartStart} yields that $M_k' \to \check{W}$ a.s. and in $L^2(\PP)$ as $k \to \infty$. Thanks to (\ref{eq:bound}) we arrive at 
\begin{align*}
\EE((Z^n)^2)&= n^{-1} a_n^2 \EE \left( \left| \frac{\check{S}_n^{(b_n)}}{a_n}-\check{W}^{(b_n)} \right|^2 \right) \\ 
&= n^{-1}a_n^2 \sum_{k=n}^\infty \EE\left( (M_k'-M_{k-1}')^2 \right) \sim \frac{(\zeta^{(b_n)})^2}{2p-1} \xrightarrow{n \to \infty} 0.
\end{align*}
This shows that $Z^n \to 0$ in $L^2(\PP)$ as $n \to \infty$ and hence also in distribution. This concludes the proof of Theorem \ref{thm:MyMainResult} for dimension $d=1$ by an appeal to (\ref{eq:decomposition}). 

\subsection{Reduction to dimension one} \label{subsection:ReductionTod1}

Recall that so far we have explicitly worked in dimension $d=1$. As we have already established Theorem \ref{thm:MyMainResult} in the one-dimensional case, we show in this section how the general $d$-dimensional case for a $d \geq 2$ follows. 

 We consider
\begin{align*}
\mathbf{ \hat{S}}_n = \mathbf{ \hat{X}}_1 + \dots + \mathbf{ \hat{X}}_n, \quad n \in \mathbb{N},
\end{align*}
with typical step distributed as $\mathbf{X}=(X^1, \dots , X^d)^\text{T} \in L^2(\PP)$. 
We observe that for each $i \in \mathbb{N}$ we have 
\begin{align*}
\mathbf{ \hat{X}}_i = ( \hat{X}_i^1, \dots , \hat{X}_i^d)^\text{T}
\end{align*}
and thus
\begin{align*}
\mathbf{ \hat{S}}_n = ( \hat{S}_n^1, \dots , \hat{S}_n^d)^\text{T}
\end{align*}
where each component $\hat{S}_n^1, \dots , \hat{S}_n^d$ of $\mathbf{\hat{S}}_n$ is itself a one-dimensional step-reinforced random walk. 

Since the almost sure convergence of a $d$-dimensional random vector is characterized by the almost sure convergence of its components, we conclude that as $n \to \infty$,
\begin{align*}
\frac{\mathbf{ \hat{S}}_n}{n^p} \to \mathbf{L} \quad \text{a.s.,}
\end{align*}
where $\mathbf{L}=(L^1, \dots , L^d) \in L^2(\PP)$ is a non-degenerate random vector in $\mathbb{R}^d$. 

Next, let $ \mathbf{a}=(a^1, \dots, a^d) \in \mathbb{R}^d$ be an arbitrary deterministic vector. Let $\langle  \cdot , \cdot \rangle$ denote the standard scalar product in $\mathbb{R}^d$. By our discussion above, it is then immediate that $\langle  \mathbf{a},\mathbf{ \hat{S}}_n \rangle$ is a one-dimensional step-reinforced random walk with typical step distributed as $\langle \mathbf{a} , \mathbf{X} \rangle \in L^2(\PP)$, observe that $\EE( \langle \mathbf{a}, \mathbf{X} \rangle)= \mathbf{0} \in \mathbb{R}^d$. Since Theorem \ref{thm:MyMainResult} has been established for dimension $d=1$, we have as $n \to \infty$ the convergence in distribution 
\begin{align} \label{conv:Distributional}
\frac{\langle \mathbf{a}, \mathbf{\hat{S}}_n \rangle - \langle \mathbf{a}, a_n \mathbf{L} \rangle }{\sqrt{n}} \implies \mathcal{N}(0, \sigma_a^2/(2p-1)),
\end{align}
where $\sigma_a^2= \EE ( \langle \mathbf{a}, \mathbf{X} \rangle^2)$. Next let us consider $\mathbf{N} \sim \mathcal{N}_d(\mathbf{0}, \bm{\sigma}^2/(2p-1))$, by definition of the multivariate normal distribution that means $\mathbf{N}= \mathbf{AZ}$ where $\mathbf{Z}=(Z^1, \dots , Z^d)^\text{T}$ with $Z^1, \dots , Z^d$ i.i.d. $\mathcal{N}(0,1)$-distributed random variables and $\mathbf{A}$ is a $d\times d$ matrix such that $\mathbf{A}\mathbf{A}^\text{T}= \bm{\sigma}^2/(2p-1)$ is the covariance matrix. We then conclude from the convergence in distribution (\ref{conv:Distributional}) that as $n \to \infty$ \begin{align*}
\frac{\langle \mathbf{a}, \mathbf{\hat{S}}_n \rangle - \langle \mathbf{a}, a_n \mathbf{L} \rangle }{\sqrt{n}} \implies \langle \mathbf{a}, \mathbf{AZ} \rangle = \langle \mathbf{a} , \mathbf{N} \rangle.
\end{align*}
By the Cramér-Wold theorem, see for example Theorem 29.4 on page 383 in
Billingsley \cite{Billingsley}, we conclude that 
\begin{align*}
\frac{\mathbf{ \hat{S}}_n - a_n \mathbf{ \hat{W}}}{\sqrt{n}}\implies \mathcal{N}_d(\bm{0}, \bm{\sigma}^2/(2p-1)),
\end{align*}
where we set $\mathbf{ \hat{W}}:= \mathbf{L}.$ Thus the proof of Theorem \ref{thm:MyMainResult} is now complete.
\section{An application: Reinforced empirical processes} \label{section:EmpiricalProcesses}
In this section we write $\mathbb{D}$ for the space of RCLL processes $\omega : [0,1] \to \mathbb{R}$ endowed with the Skorokhod topology (see Chapter 3 in \cite{Billingsley} or Chapter VI in \cite{Processes}). The notation $\Longrightarrow_\mathbb{D}$ is then reserved to indicate convergence in distribution of a sequence of processes in $\mathbb{D}$, we still use the notation $\implies$ to denote convergence in distribution as in Theorem \ref{thm:MyMainResult}.

Recall that if $U_1, U_2, \dots $ is a sequence of i.i.d. uniform random variables on the unit interval $[0,1]$; then we have the sequence of (uniform) empirical processes given by
\begin{align} \label{BridgeOne}
G_n(x):= \frac{1}{\sqrt{n}} \sum_{i=1}^n \left( \mathbf{1}_{U_i \leq x} - x \right), \quad x \in [0,1]. 
\end{align}
In 1952 Donsker \cite{Donsker} established that as $n$ tends to infinity, there is the convergence in the sense of Skorokhod 
\begin{align*}
\left( G_n(x) \right)_{x \in [0,1]} \Longrightarrow_\mathbb{D} \left( G(x) \right)_{x \in [0,1]}
\end{align*}
where $G$ denotes a Brownian bridge.

We consider $\hat{U}_1, \hat{U}_2, \dots$ to be the reinforced random variables associated to the sequence of i.i.d. uniform random variables $U_1, U_2, \dots$ on the interval $[0,1]$. We are chiefly interested in the empirical processes $\hat{G}_n$ associated to the reinforced sequence $(\hat{U}_n)_{n \geq 1}$, i.e. 
\begin{align} \label{BridgeTwo}
\hat{G}_n(x) = \frac{1}{\sqrt{n}} \sum_{i=1}^n \left( \mathbf{1}_{ \hat{U}_i \leq x} - x \right),  \quad x \in [0,1].
\end{align}
According to Theorem 2.1 in \cite{KallenbergCanonical}, the processes given by (\ref{BridgeOne}) and (\ref{BridgeTwo}) fall into the framework of bridges with exchangable increments, see Kallenberg \cite{KallenbergCanonical} and \cite{KallenbergPath} for more details. For our purpose we require the following bridge with exchangable increments as given by Definition 2.3 in \cite{BertoinLinear}; Let $p>1/2$, we define $B^{(p)}=(B^{(p)})_{x \in [0,1]}$ by
\begin{align} \label{BridgeBertoin}
B^{(p)}(x)= \sum_{j=1}^\infty X_j^{(p)} \left( \mathbf{1}_{U_j \leq x}-x \right), \quad x \in [0, 1],
\end{align}
where $(U_j)_{j \geq 1}$ is a sequence of i.i.d. uniform random variables on $[0,1]$, independent of $\mathbf{X}^{(p)}=(X_j^{(p)})_{j \geq 1}$. The sequence $\mathbf{X}^{(p)}$ arises as limits of $N_j(n)= \#\{ k \leq n : \hat{U}_k= U_j\}$, i.e. the occurrences of the variable $U_j$ up to the $n$-th step, see Lemma 2.2 in \cite{BertoinLinear} for details.

In a recent article, Bertoin \cite{BertoinLinear} studied how linear reinforcement affects empirical processes as displayed in (\ref{BridgeTwo}). Theorem 1.2 in \cite{BertoinLinear} gives functional limit theorems for all regimes $p \in (0,1)$.  Notably, for the superdiffusive regime $p>1/2$, we have 

\begin{align} \label{conv:ProbabilityinD}
\lim_{n \to \infty} n^{-p+1/2} \hat{G}_n = B^{(p)} \quad \text{in probability in } \mathbb{D}.
\end{align}
The convergence displayed in (\ref{conv:ProbabilityinD}) makes it plausible to investigate second order weak limit theorems. As an application of Theorem \ref{thm:MyMainResult} we  establish the following refinement of (\ref{conv:ProbabilityinD}):

\begin{cor} \label{thm:MySecondMainResult} Let $p \in (1/2,1)$. Then we have the convergence in the sense of Skorokhod in $\mathbb{D}$ as $n \to \infty$ 
\begin{align*}
\left(  \hat{G}_n(x) - a_n n^{-1/2} B^{(p)}(x)  \right)_{x \in [0,1]} \Longrightarrow_\mathbb{D} \frac{1}{\sqrt{2p-1}} \left( G(x) \right)_{x \in [0,1]}
\end{align*}
where $G=(G(x))_{x \in [0,1]}$ is a Brownian bridge and $ ( B^{(p)}(x))_{x \in [0,1]}$ is the bridge with exchangeable increments given by (\ref{BridgeBertoin}).
\end{cor}

\begin{proof}
We observe, with an appeal to Theorem 2.1 in \cite{KallenbergCanonical}, that for $x \in [0,1]$ the process 
\begin{align} \label{MyBridge}
\hat{G}_n(x)- a_n n^{-1/2} B^{(p)}(x)
\end{align} is again a bridge with exchangable increments. For the framework of bridges with exchangable increments, Theorem 2.3 and conditions (C) and (D) in \cite{KallenbergCanonical} shows that in order to establish the convergence in the sense of Skorokhod in $\mathbb{D}$ as dictated by Corollary \ref{thm:MySecondMainResult}, it suffices to show the convergence of the finite-dimensional distributions of (\ref{MyBridge}) and characterise the limiting random vector.

For $k \in \mathbb{N}$ let $x_1, \dots ,  x_k \in [0, 1]$ such that $x_1 \leq \dots \leq x_k$ and let $U_1, U_2, \dots $ be a sequence of i.i.d. copies of a uniform random variable on $[0,1]$ denoted by $U$. We consider for $i=1, \dots , n$ the $k$-dimensional steps given by
\begin{align*}
\mathbf{ X}_i := \left( \mathbf{1}_{ U_i \leq x_1}-x_1, \dots , \mathbf{1}_{ U_i \leq x_k} - x_k \right).
\end{align*} Plainly $\mathbf{X}_1, \mathbf{X}_2, \dots$ is a sequence of i.i.d. copies of the $\mathbb{R}^k$-valued random vector 
\begin{align*}
\mathbf{X} = \left( \mathbf{1}_{ U \leq x_1}-x_1, \dots , \mathbf{1}_{ U \leq x_k} - x_k \right).
\end{align*}
We observe that $\mathbf{X} \in L^2(\PP)$ with $\EE( \mathbf{X})= \mathbf{0}$ and we shall denote by $\bm{\sigma}^2= \EE( \mathbf{X} \mathbf{X}^\text{T})$ the covariance matrix of $\mathbf{X}$. Since we assume that $x_1  \leq \dots \leq x_k$ one easily verifies that the entries of the covariance matrix $\bm{\sigma}^2$ are given by 
\begin{align} \label{covarianceStructure}
\sigma^2_{i,j} = \begin{cases}  x_i(1-x_j) & \text{if }1 \leq i \leq j \leq k, \\
x_j(1-x_i) & \text{if } 1 \leq j < i \leq k 
\end{cases}.
\end{align} We consider the reinforced sequence $\hat{U}_1, \hat{U}_2, \dots$ and associated steps 
\begin{align*}
\mathbf{ \hat{X}}_i := \left( \mathbf{1}_{ \hat{U}_i \leq x_1}-x_1, \dots , \mathbf{1}_{ \hat{U}_i \leq x_k} - x_k \right)
\end{align*}
and we set
\begin{align*}
\mathbf{ \hat{S}}_n = \mathbf{ \hat{X}}_1 + \dots + \mathbf{ \hat{X}}_n.
\end{align*}
The process $( \mathbf{ \hat{S}}_n)_{n \in \mathbb{N}}$ is then a $k$-dimensional step-reinforced random walk. Since $p>1/2$, we have thanks to Theorem 1.2 iii) in \cite{BertoinLinear} the convergence in probability of the finite-dimensional distributions as $n$ tends to infinity
\begin{align*}
n^{1/2-p} \left( \hat{G}_n(x_1), \dots \hat{G}_n(x_k) \right) = n^{1/2-p} \frac{1}{\sqrt{n}} \mathbf{ \hat{S}}_n = \frac{\mathbf{ \hat{S}}_n}{n^p} \\
\to \left( B^{(p)}(x_1), \dots , B^{(p)}(x_k) \right).
\end{align*}
Moreover, by Theorem \ref{thm:MyMainResult} we obtain the convergence in distribution as $n$ tends to infinity 
\begin{align*}
\frac{\mathbf{ \hat{S}}_n -a_n\left( B^{(p)}(x_1), \dots , B^{(p)}(x_k) \right) }{\sqrt{n}} \implies \frac{1}{\sqrt{2p-1}} \mathcal{N}_d(\bm{0}, \bm{\sigma}^2),
\end{align*}
where we identified $\mathbf{\hat{W}}= \left( B^{(p)}(x_1), \dots , B^{(p)}(x_k) \right) $, or equivalently
\begin{align*}
\left( \hat{G}_n(x_1)- a_n n^{-1/2} B^{(p)}(x_1), \dots , \hat{G}_n(x_k)- a_n n^{-1/2} B^{(p)}(x_k) \right) \\ \implies \frac{1}{\sqrt{2p-1}} \mathcal{N}_d(\bm{0}, \bm{\sigma}^2). 
\end{align*}
Thus we have established the convergence of the finite-dimensional distributions to a Gaussian process and further identified its covariance structure via (\ref{covarianceStructure}). Since the covariance structure agrees with the one of a Brownian bridge this concludes the proof. 
\end{proof}
\section*{Acknowledgements}
I would like to express my gratitude to both of my supervisors, Jean Bertoin and Erich Baur, for their patience and valuable input while I was writing this article.
\bibliographystyle{plainurl}
\bibliography{bib}

\begin{thebibliography}{10}

\bibitem{BaurClass}
Erich Baur.
\newblock On a class of random walks with reinforced memory.
\newblock {\em Journal of Statistical Physics}, 181(3):772--802, 2020.
\newblock \href {https://doi.org/10.1007/s10955-020-02602-3}
  {\path{doi:10.1007/s10955-020-02602-3}}.

\bibitem{BaurBertoin}
Erich Baur and Jean Bertoin.
\newblock Elephant random walks and their connection to {P}\'olya-type urns.
\newblock {\em Phys. Rev. E}, 94, 2016.
\newblock \href {https://doi.org/10.1103/PhysRevE.94.052134}
  {\path{doi:10.1103/PhysRevE.94.052134}}.

\bibitem{Bercu}
Bernard Bercu.
\newblock A martingale approach for the elephant random walk.
\newblock {\em Journal of Physics A: Mathematical and Theoretical}, 51(1),
  2017.
\newblock \href {https://doi.org/10.1088/1751-8121/aa95a6}
  {\path{doi:10.1088/1751-8121/aa95a6}}.

\bibitem{BercuLaulinCenter}
Bernard Bercu and Lucile Laulin.
\newblock On the center of mass of the elephant random walk.
\newblock {\em Stochastic Processes and their Applications}, 2020.
\newblock \href {https://doi.org/10.1016/j.spa.2020.11.004}
  {\path{doi:10.1016/j.spa.2020.11.004}}.

\bibitem{Bertenghi}
Marco Bertenghi.
\newblock Functional limit theorems for the multi-dimensional elephant random
  walk.
\newblock 2020.
\newblock \href {http://arxiv.org/abs/2004.02004} {\path{arXiv:2004.02004}}.

\bibitem{BertoinLinear}
Jean Bertoin.
\newblock How linear reinforcement affects {D}onsker's theorem for empirical
  processes.
\newblock {\em Probability Theory and Related Fields}, 178:1173--1192, 2020.

\bibitem{BertoinUniversality}
Jean Bertoin.
\newblock Universality of noise reinforced {B}rownian motions.
\newblock {\em Arxiv Preprints}, 2020.
\newblock URL: \url{https://hal.archives-ouvertes.fr/hal-02341310}.

\bibitem{Billingsley}
Patrick Billingsley.
\newblock {\em Probability and measure}.
\newblock John Wiley \& Sons, 2008.

\bibitem{GavaSchuetzColetti}
Cristian~F. Coletti, Renato Gava, and Gunter~M. Schütz.
\newblock Central limit theorem and related results for the elephant random
  walk.
\newblock {\em Journal of Mathematical Physics}, 58, 2017.
\newblock \href {https://doi.org/10.1063/1.4983566}
  {\path{doi:10.1063/1.4983566}}.

\bibitem{Coletti2017}
Cristian~F Coletti, Renato Gava, and Gunter~M Schütz.
\newblock A strong invariance principle for the elephant random walk.
\newblock {\em Journal of Statistical Mechanics: Theory and Experiment},
  2017(12):123207, dec 2017.
\newblock \href {https://doi.org/10.1088/1742-5468/aa9680}
  {\path{doi:10.1088/1742-5468/aa9680}}.

\bibitem{Coletti2019}
Cristian~F. Coletti and Ioannis Papageorgiou.
\newblock Asymptotic analysis of the elephant random walk.
\newblock 2019.
\newblock \href {http://arxiv.org/abs/1910.03142} {\path{arXiv:1910.03142}}.

\bibitem{Donsker}
Monroe~D Donsker.
\newblock Justification and extension of {D}oob's heuristic approach to the
  {K}olmogorov-{S}mirnov theorems.
\newblock {\em The Annals of mathematical statistics}, pages 277--281, 1952.
\newblock \href {https://doi.org/10.1214/aoms/1177729445}
  {\path{doi:10.1214/aoms/1177729445}}.

\bibitem{GonzalesMult}
Manuel González-Navarrete.
\newblock Multidimensional walks with random tendency.
\newblock {\em Journal of Statistical Physics volume}, 181:1138--1148, 2020.
\newblock \href {https://doi.org/10.1007/s10955-020-02621-0}
  {\path{doi:10.1007/s10955-020-02621-0}}.

\bibitem{GuevaraERW}
V{\i}ctor Hugo~V{\'a}zquez Guevara and Hugo~Cruz Su{\'a}rez.
\newblock An elephant random walk based strategy for improving learning.

\bibitem{GuevaraMinimal}
V{\i}ctor Hugo~V{\'a}zquez Guevara and Hugo~Cruz Su{\'a}rez.
\newblock A strategy to improve learning via a minimal random walk.

\bibitem{Heyde}
Christopher~C. Heyde.
\newblock On central limit and iterated logarithm supplements to the martingale
  convergence theorem.
\newblock {\em Selected Works in Probability and Statistics}, 2010.
\newblock \href {https://doi.org/10.1007/978-1-4419-5823-5_42}
  {\path{doi:10.1007/978-1-4419-5823-5_42}}.

\bibitem{Processes}
Jean Jacod and Albert~N. Shiryaev.
\newblock {\em Limit Theorems for Stochastic Processes}.
\newblock Springer, 2003.
\newblock \href {https://doi.org/10.1007/978-3-662-05265-5}
  {\path{doi:10.1007/978-3-662-05265-5}}.

\bibitem{KallenbergCanonical}
Olav Kallenberg.
\newblock Canonical representations and convergence criteria for processes with
  interchangeable increments.
\newblock {\em Zeitschrift f{\"u}r Wahrscheinlichkeitstheorie und Verwandte
  Gebiete}, 27(1):23--36, 1973.
\newblock \href {https://doi.org/10.1007/BF00736005}
  {\path{doi:10.1007/BF00736005}}.

\bibitem{KallenbergPath}
Olav Kallenberg.
\newblock Path properties of processes with independent and interchangeable
  increments.
\newblock {\em Zeitschrift f{\"u}r Wahrscheinlichkeitstheorie und Verwandte
  Gebiete}, 28(4):257--271, 1974.
\newblock \href {https://doi.org/10.1007/BF00532944}
  {\path{doi:10.1007/BF00532944}}.

\bibitem{KubotaTakei}
Naoki Kubota and Masato Takei.
\newblock Gaussian fluctuation for superdiffusive elephant random walks.
\newblock {\em Journal of Statistical Physics 177}, pages 1157--1171, 2019.
\newblock \href {https://doi.org/10.1007/s10955-019-02414-0}
  {\path{doi:10.1007/s10955-019-02414-0}}.

\bibitem{Kuersten}
R\"udiger K\"ursten.
\newblock Random recursive trees and the elephant random walk.
\newblock {\em Phys. Rev. E}, 93:032111, Mar 2016.
\newblock \href {https://doi.org/10.1103/PhysRevE.93.032111}
  {\path{doi:10.1103/PhysRevE.93.032111}}.

\bibitem{SchuetzTrimper}
Gunter~M. Sch\"utz and Steffen Trimper.
\newblock Elephants can always remember: Exact long-range memory effects in a
  non-markovian random walk.
\newblock {\em Phys. Rev. E}, 70, 2004.
\newblock \href {https://doi.org/10.1103/PhysRevE.70.045101}
  {\path{doi:10.1103/PhysRevE.70.045101}}.

\end{thebibliography}
\end{document}